\documentclass[10pt,a4paper,english]{smfart}

\usepackage[a4paper,left=4cm,right=4cm,top=4cm,bottom=4cm]{geometry}
\usepackage[english]{babel}
\usepackage{amsfonts, amsthm, amsmath, amscd, amssymb}
\usepackage[linktocpage]{hyperref}
\usepackage{graphicx}
\usepackage[all]{xy}

\newcommand{\cp}[2]{{(#1,#2)=1}}
\renewcommand\AA{\mathbb{A}}
\newcommand\LL{\mathbf{L}}
\newcommand\A{\mathcal{A}}
\newcommand\QQ{\mathbb{Q}}
\newcommand\N{\mathcal{N}}
\newcommand\FF{\mathbb{F}}
\newcommand\RR{\mathbb{R}}
\newcommand\RRnn{\RR_{\ge 0}}
\newcommand\R{\mathcal{R}}
\newcommand\T{\mathcal{T}}
\newcommand\ZZ{\mathbb{Z}}
\newcommand\ZZp{\ZZ_{>0}}
\newcommand\ZZnz{\ZZ_{\ne 0}}
\newcommand\PP{\mathbb{P}}
\newcommand\xx{\mathbf{x}}
\renewcommand\tt{\mathbf{t}}
\newcommand\dd{\,\mathrm{d}}
\newcommand\Ptwo{{\PP^2}}
\newcommand{\base}[8]{\ee^{({#1},{#2},{#3},{#4},{#5},{#6},{#7},{#8})}}
\newcommand{\congr}[3]{{#1} \equiv {#2} \bmod{#3}}
\DeclareMathOperator{\Pic}{Pic}
\DeclareMathOperator{\vol}{vol}
\DeclareMathOperator{\Spec}{Spec}
\DeclareMathOperator{\Cox}{Cox}

\DeclareMathOperator{\rad}{rad}
\newcommand{\Aone}{{\mathbf A}_1}
\newcommand{\Atwo}{{\mathbf A}_2}
\newcommand{\Afour}{{\mathbf A}_4}
\newcommand{\Afive}{{\mathbf A}_5}
\newcommand{\Dfour}{{\mathbf D}_4}
\newcommand{\Dfive}{{\mathbf D}_5}
\newcommand{\Dsix}{{\mathbf D}_6}
\newcommand{\Esix}{{\mathbf E}_6}
\newcommand{\Eseven}{{\mathbf E}_7}
\newcommand{\Eeight}{{\mathbf E}_8}
\newcommand{\tS}{{\widetilde S}}
\renewcommand{\le}{\leqslant}
\renewcommand{\ge}{\geqslant}
\renewcommand\rho{\varrho}
\newcommand{\ee}{\boldsymbol{\eta}}
\newcommand{\ex}[1]{*+<10pt>[o][F]{#1}}
\newcommand\rto{\dashrightarrow}
\newcommand\e{\eta}
\newcommand\ep{\varepsilon}
\newcommand\midd{\,\Big|\,}
\newcommand\GG{\mathbb{G}}
\newcommand\Ga{\GG_\mathrm{a}}
\newcommand\Gm{\GG_\mathrm{m}}

\newtheorem{cor}{Corollary}
\newtheorem{theorem}{Theorem}
\newtheorem{lemma}[theorem]{Lemma}

\theoremstyle{definition}
\newtheorem{remark}[theorem]{Remark}

\newtheorem*{ack}{Acknowledgements}

\numberwithin{equation}{section}

\begin{document}

\title[Quadratic congruences and rational points on cubic surfaces]
{Quadratic congruences on average and rational points on cubic surfaces}

\author{Stephan Baier}

\address{Tata Institute of Fundamental Research, 1 Dr.\ Homi Bhaba
  Road, Colaba, Mumbai 400005, India}

\email{sbaier@math.tifr.res.in}

\author{Ulrich Derenthal} 

\address{Institut f\"ur Algebra, Zahlentheorie und Diskrete
  Mathematik, Leibniz Universit\"at Hannover,
  Welfengarten 1, 30167 Hannover, Germany}

\email{derenthal@math.uni-hannover.de}

\begin{abstract}
  We investigate the average number of solutions of certain quadratic
  congruences. As an application, we establish Manin's conjecture for
  a cubic surface whose singularity type is $\Afive+\Aone$.
\end{abstract}

\keywords{Quadratic congruences, rational points, Manin's conjecture, cubic surfaces, universal torsors}

\subjclass{11D45 (14G05, 11G35)}

%
%

\date{July 22, 2015}

\maketitle

\tableofcontents

\section{Introduction}

Given a (possibly singular) del Pezzo surface $S$ defined over the
field $\QQ$ of rational numbers and containing infinitely many
rational points, we would like to study the distribution of these
points more precisely. We will be most interested in the cubic surface
of singularity type $\Afive+\Aone$ defined in $\PP^3$ by
\begin{equation}\label{eq:surface}
  x_1^3+x_2x_3^2+x_0x_1x_2 = 0.
\end{equation}

Let $H : S(\QQ) \to \RR$ be an anticanonical height function. The
number of rational points of bounded height on $S$ is dominated by the number of
points lying on the lines on (an anticanonical model of) $S$. Therefore, it is more
interesting to study rational points of height bounded by $B$ on the
complement $U$ of the lines on $S$, i.e., the number
\begin{equation*}
  N_{U,H}(B) = \#\{\xx \in U(\QQ) \mid H(\xx) \le B\}.
\end{equation*}

Manin's conjecture \cite{MR89m:11060} predicts that, as $B$ tends to $+ \infty$,
\begin{equation*}
  N_{U,H}(B) = c_{S,H}B(\log B)^{r-1} (1+o(1)),
\end{equation*}
where $r$ is the rank of the Picard group of (a minimal desingularization of)
$S$ and $c_{S,H}$ is a positive constant for which Peyre, Batyrev and
Tschinkel have given a conjectural interpretation \cite{MR1340296},
\cite{MR1679843}.

\begin{figure}[ht]
  \centering
  \includegraphics[width=11cm]{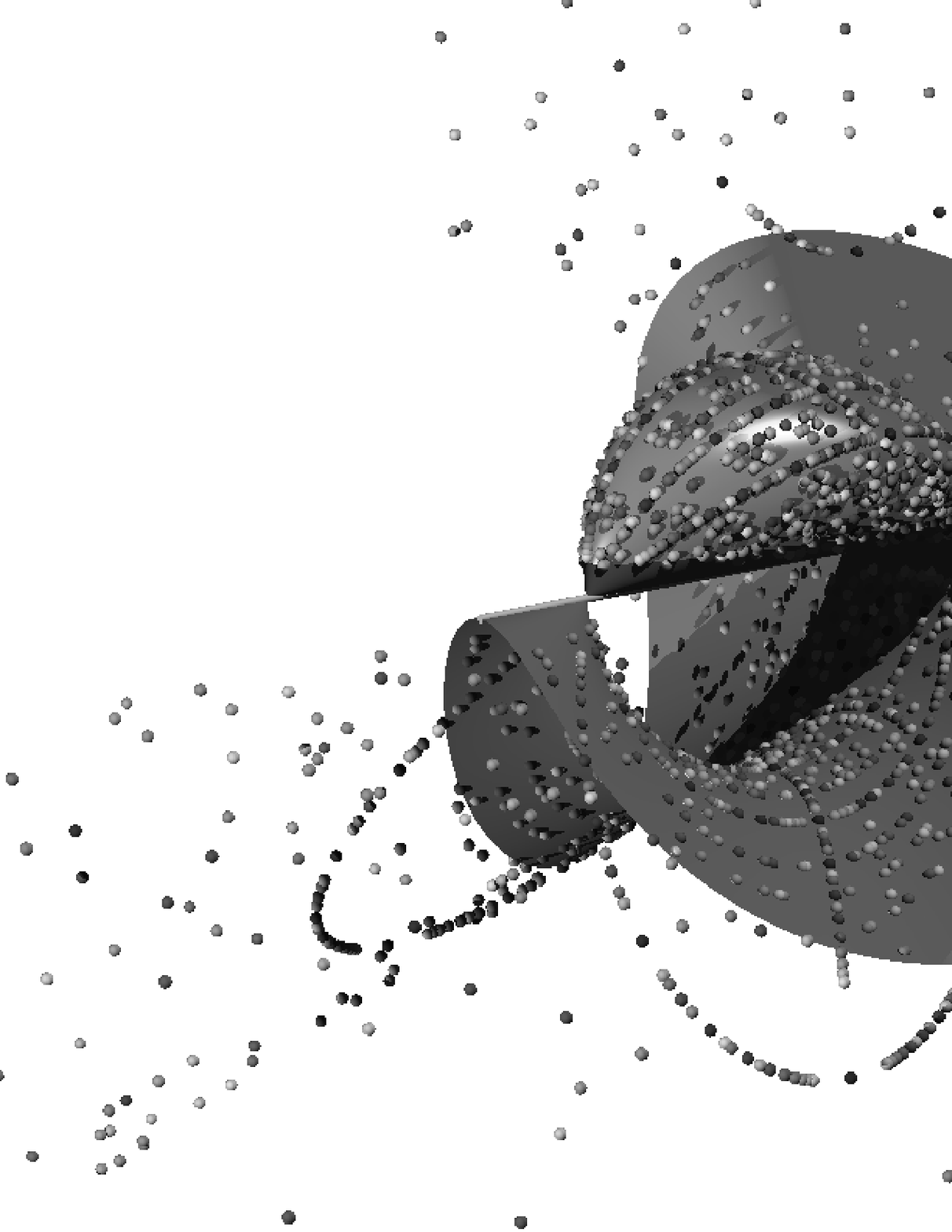}
  \caption{Points of height at most $100$ on the $\Afive+\Aone$ cubic surface.}
  \label{fig:d5}
\end{figure}

If $S$ is an equivariant compactification of an algebraic group $G$,
Manin's conjecture can be proved in certain cases. For instance, see
\cite{MR1620682} for the case of toric varieties (with $G = \Gm^2$),
\cite{MR1906155} for the case of the additive group $G=\Ga^2$ and
\cite{MR2858922} for certain semidirect products $G = \Ga \rtimes
\Gm$. However, the equation (\ref{eq:surface}) defines a cubic surface
that is not covered by any of these results (see \cite{MR2753646}, \cite{MR3333982}).

For general surfaces $S$, one can approach Manin's conjecture resorting to
universal torsors. Using Cox rings, a universal torsor $\T$ of a
minimal desingularization $\tS$ of a del Pezzo surface $S$ of degree
$d$ can be explicitly described as an open subset of an affine
variety $\Spec\Cox(\tS)$. The basic case is again the one of toric
varieties \cite{MR1679841}, where $\Spec\Cox(\tS) \cong \AA^{12-d}$ is an
affine space.

The next natural case is the one where $\Spec\Cox(\tS) \subset \AA^{13-d}$ is
a hypersurface, defined by one \emph{torsor equation} in the variables $\e_1,
\dots, \e_{13-d}$. For example, for our surface of degree $d=3$ and type
$\Afive+\Aone$, the torsor equation is
\begin{equation}\label{eq:torsor1}
  \e_1\e_{10} + \e_2\e_9^2+\e_4\e_5^2\e_6^4\e_7^3\e_8 = 0.
\end{equation}
All such del Pezzo surfaces are classified in \cite{MR3180592},
where a detailed description of $\Cox(\tS)$ is also given.

The passage to a universal torsor translates the problem of counting rational
points on $S$ to the one of counting tuples $(\e_1, \dots, \e_{13-d})$ of
integers satisfying the torsor equation and certain height and coprimality
conditions.

This is basically done as follows. The coprimality conditions can be taken
care of by M\"obius inversions (in this introduction, we will simply ignore
all auxiliary variables occuring because of this). Using a torsor equation
such as~(\ref{eq:torsor1}), we may eliminate one variable $\e_{13-d}$ that
occurs linearly in it. Fixing $\e_1, \dots, \e_{11-d}$, we are led to counting
the number of integers $\e_{12-d}$ satisfying a congruence condition modulo
some integer $q$ and lying in some range $I$ given by the height
conditions. In our example, the congruence condition is
\begin{equation*}
  \congr{\e_2\e_9^2}{-\e_4\e_5^2\e_6^4\e_7^3\e_8}{\e_1}.
\end{equation*}
Note that both $I$ and $q$ may depend on $\e_1, \dots, \e_{11-d}$.

If $\e_{12-d}$ also occurs linearly in the torsor equation then the congruence
is linear, so that the number of such $\e_{12-d}$ is basically
$q^{-1}\vol(I)+E$, where $E=O(1)$. Summing this over the remaining variables
$\e_1, \dots, \e_{11-d}$, we must estimate the main term $q^{-1}\vol(I)$ and
show that the contribution of the error term $E$ is negligible. The estimation
of the error term of the first summation is sometimes straightforward and
sometimes very hard. The estimation of the main term is expected to be often
straightforward using the results of \cite[Sections~4, 5, 7]{MR2520770} in
case of linear $\e_{12-d}$.

However, if $\e_{12-d}$ occurs with a square power in the torsor
equation (such as $\e_9^2$ in~\eqref{eq:torsor1}), the main term
contains an extra factor of the shape
\begin{equation}\label{eq:rho-term}
\N(a,q) = \#\{\rho \mid 1 \le \rho \le q,\ \cp{\rho}{q},\ \congr{\rho^2}{a}{q}\},
\end{equation}
where $a$ and $q$ are, basically, monomials in $\e_1, \dots, \e_{11-d}$ (for
instance $q=\e_1$ and $a=-\e_2\e_4\e_7\e_8$ in our example; see also
\cite[Proposition~2.4]{MR2520770}).  Our experience is that the presence of
$\N(a,q)$ usually makes the treatment of the error term in the next summation
over $\e_{11-d}$ (over some interval $J$) much harder.

Following the most natural order of summation (which is guided by the
requirement to start with the $\e_i$ that may be the largest), 
a term of the shape $\N(a,q)$ appears in the treatment of
the following singular del Pezzo surfaces (with one torsor equation):
\begin{itemize}
\item quartic del Pezzo surfaces of types $\Dfive$ and $\Afour$,
\item cubic surfaces of types $\Esix$, $\Dfive$, $\Afive+\Aone$,
\item del Pezzo surfaces of degree $2$ of types $\Eseven$, $\Esix$, $\Dsix+\Aone$,
\item del Pezzo surfaces of degree $1$ of types $\Eeight$,
  $\Eseven+\Aone$.
\end{itemize}

Let us sketch the effects of $\N(a,q)$ in the summation of the main term over
$\e_{11-d}$ in an interval $J$. To avoid complications which are irrelevant to
our point, we replace $q^{-1}\vol(I)$ by $1$ for the moment; this can be
restored by using partial summation.  If $\e_{11-d}$ occurs linearly in $a$,
we can switch the order of the summations over $\rho$ and $\e_{11-d}$. Then
the summation over $\e_{11-d}$ subject to the linear congruence modulo $q$
gives the main term $q^{-1}\vol(J)$ and an error term $F=O(1)$, which we must
sum over $\rho$ subject to $1 \le \rho \le q$ and $\cp{\rho}{q}$ and over the
remaining variables $\e_1, \dots, \e_{10-d}$.

The most naive estimation $\sum_{\rho=1}^q F = O(q)$ is usually not good
enough. This problem has been approached in several different ways.

\begin{itemize}
\item For the quartic $\Afour$ case \cite{MR2543667}, it is enough to
  obtain an extra saving by using different orders of summation over
  $\e_{11-d}$ and $\e_{10-d}$, depending on their relative size.
\item Alternatively, one can get an extra saving by making $F$ explicit,
  improving $O(q)$ to $O(q^{1/2+\ep})$ as in \cite[Lemma~3]{MR2320172} using
  Fourier series and quadratic Gauss sums, which is sufficient for the second
  summation for the quartic surface of type $\Dfive$ \cite{MR2320172} and for
  the cubic surface of type $\Esix$ \cite{MR2332351}; for the latter over
  imaginary quadratic fields, one can apply Poisson summation combined with
  Hua's results for exponentional sums over number fields \cite{MR3352254}.
\item For the cubic surface of type $\Dfive$ \cite{MR2520769}, the
  previous two approaches are combined and slightly improved.
\item For the degree $2$ del Pezzo surface of type $\Eseven$
  \cite{MR3100953}, the first two summations over
  $\e_{11-d},\e_{12-d}$ are treated simultaneously.
\end{itemize}
Furthermore, Manin's conjecture is known for some smooth and singular del
Pezzo surfaces of degree greater or equal to $3$ for which the factor
$\N(a,q)$ does not appear, in particular for certain singular cubic surfaces of
types $2\Atwo+\Aone$ \cite{MR2990624} and $\Dfour$ \cite{MR3263143}.

However, for other cases such as the cubic surface $S$ of type $\Afive+\Aone$,
different ideas seem to be needed. In our approach, the main novelty is that
we get cancellation effects from its summation over $\rho$, several variables
$\e_i$ occuring linearly in $a$ and, most importantly, a variable $\e_1$
occuring in $q$, while using the trivial $O(1)$-bound for $F$. This is done in
Section~\ref{sec:quadratic}, using the Polya-Vinogradov bound for character
sums and Heath-Brown's large sieve for real character sums \cite{MR1347489}.

In what follows, for $X>0$, the notation $x\sim X$ indicates that $X < x \le 2X$. Let $K_2,K_4,K_7,K_8, Q \geq 1/2$ and $K = K_2K_4K_7K_8$.
Applied to the cubic surface of type $\Afive+\Aone$, the most basic case of our result gives the asymptotic formula
\begin{equation}\label{eq:simple_case}
\sum_{\substack{\e_i \sim K_i \\ i=2,4,7,8}} \sum_{\e_1 \sim Q} \N(-\e_2\e_4\e_7\e_8,\e_1) = c K Q + O(K^{1-\delta} Q (\log Q)^{1+\varepsilon}),
\end{equation}
for some explicit $c,\delta>0$ and for any fixed $\varepsilon> 0$. 

Our result shall be compared with the work of Heath-Brown \cite[Section $5$]{MR2075628}.
In order to obtain an upper bound for $N_{U,H}(B)$ in the case of Cayley's cubic
surface, Heath-Brown proved that the left-hand side of~(\ref{eq:simple_case})
is $\ll K Q$. However, to obtain an asymptotic formula for $N_{U,H}(B)$ for the
cubic surface defined by the equation~(\ref{eq:surface}), we need an asymptotic
formula for the left-hand side of~(\ref{eq:simple_case}),
but also for the more complicated expression $\Sigma$ defined in~(\ref{expression}).

Comparing the proof of the asymptotic formula for $\Sigma$ stated in
Theorem~\ref{secondresult} and its application in Section~\ref{sec:application} with
Heath-Brown's work, we notice that our result involves several extra difficulties.
In particular, we have to isolate the main term, work out the case of even $q$,
include a weight function and some additional parameters, and finally work with ranges
for $\eta_1$ depending on the remaining variables.  This latter task is the main difficulty and its
resolution requires some extra tools such as Perron's formula.

It is also interesting to note that we essentially manage to remove the factor
$\N(a,q)$ from the main term of the first summation in
Lemma~\ref{lem:first_sum2}, so that we can continue the proof just as in the case of
linear $\e_{11-d}$ in the torsor equation.

As an application of our general estimate for the average number of solutions of our
quadratic congruence, we prove Manin's conjecture for the cubic surface $S$ of singularity type $\Afive+\Aone$
defined by the equation~(\ref{eq:surface}). The complement of the lines is
$U = S \setminus \{x_1 = 0\}$. We use the anticanonical height function defined by
$H(\xx) = \max\{|x_0|,\dots,|x_3|\}$ for $\xx=(x_0:\dots:x_3)$, where
$(x_0, \dots, x_3) \in \ZZ^4$ is such that $(x_0, \dots,x_3)=1$. See
Section~\ref{sec:geometry} for more information on the geometry of
$S$. Besides Theorem~\ref{secondresult}, our main result is as follows.

\begin{theorem}\label{thm:main}
Let $\ep>0$ be fixed. As $B$ tends to $+ \infty$, we have the estimate
  \begin{equation*}
    N_{U,H}(B) = c_{S,H}B(\log B)^6  + O(B(\log B)^{5 + \ep}),
  \end{equation*}
  where
  \begin{equation*}
    c_{S,H} = \frac{1}{172800} \cdot
      \omega_\infty \cdot \prod_p\left(1-\frac 1 p\right)^7\left(1+\frac 7 p+\frac
        1{p^2}\right),
  \end{equation*}
  and
  \begin{equation*}
    \omega_\infty =
    \int_{0\le|(x_1x_2)^{-1}(x_1^3+x_2x_3^2)|,|x_1|,x_2,|x_3|\le 1} 
    \frac{1}{x_1x_2} \dd x_1 \dd x_2 \dd x_3.
  \end{equation*}
\end{theorem}

We will check in Section~\ref{sec:compatibility} that this agrees
with Manin's conjecture and that the constant $c_{S,H}$ is the one predicted
by Peyre, Batyrev and Tschinkel.

\begin{ack}
  The first-named author was supported by an ERC grant 258713. The
  second-named author was supported by DFG grant DE 1646/2-1 and SNF
  grant 200021\_124737/1. This collaboration was started at the
  Sino-French Summer Institute in Arithmetic Geometry 2011 at the
  Chern Institute of Mathematics in Tianjin. It was also supported by
  the Center for Advanced Studies at LMU M\"unchen. The authors thank
  these institutions for their hospitality. We thank Pierre Le Boudec
  for his help with the first version of this article, and we thank
  the anonymous referee for helpful remarks.
\end{ack}

\section{Quadratic congruences on average}\label{sec:quadratic}

As explained in the introduction, our motivation to study quadratic
congruences in this section is their appearance in proofs of Manin's conjecture.

\subsection{Counting solutions of quadratic congruences} \label{counting} 

To evaluate the main term of the first summation over a variable occuring
non-linearly in the torsor equation (such as $\e_9$ in~\eqref{eq:torsor1} in
our example; see Lemma \ref{lem:first_sum} below for the result of the first
summation in our case and \cite[Proposition~2.4]{MR2520770} for the result in
a more general situation), we need to count solutions of quadratic congruences
on average. To this end, we consider the following general situation.
 
Let $b\in \ZZ\setminus\{0\}$, $k\in \ZZp$ with $(k,b)=1$, $r\in
\ZZp$ with $r\ge 2$ and $K_1,...,K_r,Q,V$ be positive real
numbers.  Throughout, for $X>0$, we use the notation $x\sim X$ to indicate that
$X < x \le 2X$. Let $b\in \ZZ\setminus\{0\}$, $k\in \ZZp$ with $(k,b)=1$, $r\in
\ZZp$ with $r\ge 2$ and $K_1,...,K_r,Q,V$ be positive real
numbers. We assume that $\Phi$ is a continuous real-valued function
defined on $(K_1,2K_1]\times\cdots \times (K_r,2K_r]\times (0,Q]$ which
satisfies
\begin{equation} \label{Vcond}
0\le \Phi\le V
\end{equation}
and, in each of the variables, can be divided into finitely many continuously
differentiable and monotone pieces whose number is bounded by an absolute constant.  We
further assume that $Q^{-}$ and $Q^+$ are continuous real-valued functions
defined on $(K_1,2K_1]\times\cdots \times (K_r,2K_r]$ such that
\begin{equation} \label{Qcond}
0<Q^-\le Q^+\le Q .
\end{equation}
Moreover, for any given $i\in \{1,...,r\}$, for $x_j \sim K_j$ for $j \in \{1,...,r\}\setminus\{i\}$, and for $0<y\le Q$, we assume that the set
\begin{equation} \label{Adef}
\begin{split}
& \mathcal{A}_i(x_1,...,x_{i-1},x_{i+1},...,x_r,y)\\ ={} & \left\{x_i \sim K_i \ |\ Q^{-}(x_1,...,x_r)<y\le Q^{+}(x_1,...,x_r)\right\} 
\end{split}
\end{equation}
is the union of finitely many intervals whose number is bounded by an absolute
constant. Throughout the sequel, for brevity, we write
\begin{equation} \label{Kdef}
K=2^{r+1}K_1\cdots K_r,
\end{equation}
\begin{equation*}
Q^{\pm}=Q^{\pm}(a_1,...,a_r),
\end{equation*}
and
\begin{equation} \label{Anotation}
\mathcal{A}_i(y)=\mathcal{A}_i(x_1,...,x_{i-1},x_{i+1},...,x_r,y).
\end{equation}
Finally, for any integer $n \in \ZZp$, we set
\begin{equation} \label{k1def}
\rad(n)=\prod\limits_{p|n} p.
\end{equation}

Our goal is to evaluate asymptotically the expression
\begin{equation} \label{expression}
  \Sigma =\sum\limits_{a_1\sim K_1} \cdots \sum\limits_{a_r \sim K_r}\
  \sum\limits_{Q^{-}<q\le Q^{+}} \Phi(a_1,...,a_r,q) \N(-a_1\cdots a_rb,kq),
\end{equation}
where $\N(-a_1\cdots a_rb,kq)$ is defined in~(\ref{eq:rho-term}).
 
We begin by splitting $\Sigma$ into a main term and an error term.  Let $kq=2^{v(kq)}h$,
where $v(\ell)$ is the 2-adic valuation of $\ell \in \ZZp$ and $h$ is odd.
Thus, for any $n \in \ZZ$, we have
\begin{equation} \label{splitting1}
\sum\limits_{\congr{\rho^2}{n}{kq}} 1 =\left(\sum\limits_{\congr{\rho^2}{n}{2^{v(kq)}}} 1 \right)\left(\sum\limits_{\congr{\rho^2}{n}{h}} 1\right).
\end{equation}
In the following, for $j \geq 0$, we set
\begin{equation*}
\left\{\frac{n}{2^j}\right\} = \sum\limits_{\substack{\rho \bmod{2^j}\\  \congr{\rho^2}{n}{2^j}}} 1.
\end{equation*}
It is well-known that if $(n,2^j)=1$, then
\begin{equation} \label{2symbol}
\left\{\frac{n}{2^j}\right\}=
\begin{cases}
1 & \mbox{if } j=0,\\
1 & \mbox{if } \congr{n}{1}{2} \mbox{ and } j=1,\\
2 & \mbox{if } \congr{n}{1}{4} \mbox{ and } j=2,\\
4 & \mbox{if } \congr{n}{1}{8} \mbox{ and } j\ge 3,\\
0 & \mbox{otherwise.}
\end{cases}
\end{equation}
Moreover, if $h$ is odd and $(n,h)=1$, then
\begin{equation} \label{quadcon}
\sum\limits_{\congr{\rho^2}{n}{h}} 1=\sum\limits_{d|h} \mu^2(d)\left(\frac{n}{d}\right).
\end{equation}
The equalities \eqref{splitting1}, \eqref{2symbol} and \eqref{quadcon} imply that if $(a_1\cdots a_rb,kq)=1$ then
\begin{equation}\label{Jacobi} 
  \N(-a_1\cdots a_rb,kq)
  =\left\{\frac{-a_1\cdots a_rb}{2^{v(kq)}}\right\}
  \sum\limits_{\substack{d|kq\\ (d,2)=1}} \mu^2(d)\left(\frac{-a_1\cdots a_rb}{d}\right).
\end{equation}
If $(a_1 \cdot a_rb,kq)\not=1$, then $ \N(-a_1\cdots a_rb,kq)=0$.
Therefore, we deduce that we can write
\begin{equation} \label{splitting}
\Sigma=M+E,
\end{equation}
where the main term $M$ is defined by
\begin{equation} \label{main}
M=\mathop{\sum\limits_{a_1\sim K_1}  \cdots \sum\limits_{a_r \sim K_r}\ \sum\limits_{Q^{-}<q\le 
Q^{+}} }_{(a_1\cdots a_rb,kq)=1} \Phi(a_1,...,a_r,q)\left\{\frac{-a_1\cdots a_rb}{2^{v(kq)}}\right\},
\end{equation}
and the error term $E$ is defined by
\begin{equation} \label{error}
\begin{split}
E={}& \mathop{\sum\limits_{a_1\sim K_1}  \cdots \sum\limits_{a_r \sim K_r}\ \ \sum\limits_{Q^{-}<q\le 
Q^{+}} }_{(a_1\cdots a_rb,kq)=1} \Phi(a_1,...,a_r,q) \left\{\frac{-a_1\cdots a_rb}{2^{v(kq)}}\right\} \times \\ & \sum\limits_{\substack{d|kq\\ d>1\\ (d,2)=1}} \mu^2(d)\left(\frac{-a_1\cdots a_rb}{d}\right).
\end{split}
\end{equation}
In the following sections, we estimate the error term by generalizing the method used by Heath-Brown in \cite[Section~5]{MR2075628}. We shall not evaluate the main term any further since this is not needed in our application. Our result is as follows.

\begin{theorem}\label{secondresult}
Let $\ep > 0$ be fixed. Set $L=\log(2+Q)$. We have the estimate
\begin{equation*} 
\Sigma - M \ll E',
\end{equation*}
where 
\begin{equation*}
E' = V K^{1/2+\varepsilon} Q L^{\varepsilon}
    \left( K^{1/2-1/2r}\rad(k)^{1/4} + |b|^{\varepsilon}2^{(1+\varepsilon)\omega(k)}+ 2^{\omega(k)}L \right) . 
\end{equation*}
\end{theorem}

The term $\Sigma$ is not exactly the one that we need in our application. 
Let $\Sigma'$ be defined like $\Sigma$ in \eqref{expression}, but with some additional coprimality conditions included, namely
\begin{equation} \label{S'}
  \Sigma'= \mathop{\sum\limits_{\substack{a_1\sim K_1\\ (a_1,t_1)=1}} \cdots
    \sum\limits_{\substack{a_r \sim K_r\\ (a_r,t_r)=1}}}_{(a_i,a_j)=1, \ 1 \leq i < j \leq r}
  \sum\limits_{\substack{Q^{-}<q\le Q^{+}\\ (q,u)=1}} \Phi(a_1,...,a_r,q)
  \N(-a_1\cdots a_rb,kq),
\end{equation}
where $t_1,...,t_r,u\in \ZZp$. Accordingly, we set
\begin{equation}\label{main'}
  M'=\mathop{\sum\limits_{\substack{a_1\sim K_1\\
        (a_1,t_1)=1}} \cdots \sum\limits_{\substack{a_r \sim K_r\\
        (a_r,t_r)=1}}}_{(a_i,a_j)=1, \ 1 \leq i < j \leq r} \sum\limits_{\substack{Q^{-}<q\le Q^{+}\\ (q,u)=1\\
      (a_1\cdots a_rb,kq)=1}} \Phi(a_1,...,a_r,q)\left\{\frac{-a_1\cdots
      a_rb}{2^{v(kq)}}\right\}.
\end{equation}
Removing the additional coprimality conditions using M\"obius inversions,  we shall deduce  from Theorem \ref{secondresult} the following asymptotic formula for $\Sigma'$.

\begin{cor}\label{secondresultcorollary}
Let $\ep > 0$ be fixed. We have the estimate
\begin{equation*}
\Sigma' - M' \ll (1+\varepsilon)^{\omega(t_1)+\cdots +\omega(t_r)+\omega(u)} E'.
\end{equation*}
\end{cor}

\begin{remark} \label{intervalrem} Theorem \ref{secondresult} and Corollary \ref{secondresultcorollary} remain true if the left half-open $q$-summa\-tion interval $(Q^-,Q^+]$ is replaced by an arbitrary interval $\mathcal{I}(Q^-,Q^+)$ (left half-open, right half-open, open, closed) with endpoints $Q^-$ and $Q^+$.  The proof is the same, with the relevant summation intervals being altered accordingly.
\end{remark}

Theorem \ref{secondresult} and Corollary \ref{secondresultcorollary} trivially hold if
$K_i<1/2$ for some $i\in \{1,...,r\}$ or $Q<1$ since in this case we
have $\Sigma=M=0$. Therefore, we shall assume that $K_i\ge 1/2$ for any $i\in
\{1,...,r\}$ and $Q\ge 1$ throughout the following proofs of these
results.  
Therefore, recalling the definition \eqref{Kdef} of $K$, we note that $K \geq 2$. 

\subsection{Application of the Polya-Vinogradov bound I} \label{Polvin}
Let us write $d=fg$, where $g=(d,k)$. It follows that $(f,k/g)=1$ and so the condition $d|kq$ is equivalent to $f|q$. Thus, we can write $q=ef$.  Let us set
\begin{equation*}
\begin{split}
Q^-(e,g) & =  \max\{1/g,Q^{-}/e\}, \\
Q^+(e) & =  Q^+/e.
\end{split}
\end{equation*}
Reordering the summations and noting that $\mu^2(fg)=1$ if and only if $(f,g)=1$ and $\mu^2(f)=\mu^2(g)=1$,
the error term $E$ defined in \eqref{error} can be rewritten as
\begin{equation} \label{errorrewritten}
E=\sum\limits_{\substack{g|k\\ (g,2)=1}} \mu^2(g) \sum\limits_{\substack{e\le Q\\ (e,b)=1}} E(e,g), 
\end{equation}
where 
\begin{equation} \label{Segdef}
\begin{split}
E(e,g)={}&  \mathop{\sum\limits_{a_1\sim K_1}  \cdots \sum\limits_{a_r \sim K_r}}_{(a_1\cdots a_r,ke)=1} \left\{\frac{-a_1\cdots a_rb}{2^{v(ke)}}\right\}\sum\limits_{\substack{Q^-(e,g)<f\le 
Q^{+}(e)\\ (f,2k)=1}} \Phi(a_1,...,a_r,ef)\times\\ &  \mu^2(f)\left(\frac{-a_1\cdots a_rb}{fg}\right).
\end{split}
\end{equation}

In the following sections, we will estimate $E(e,g)$ in three different
ways. We start with an application of the Polya-Vinogradov bound for
character sums.  Pulling in the summation over $a_1$, we get
\begin{equation} \label{Seg}
\begin{split}
E(e,g)={}& \mathop{\sum\limits_{a_2\sim K_2}  \cdots \sum\limits_{a_r \sim K_r}}_{(a_2\cdots a_r,ke)=1}\  \sum\limits_{\substack{1/g<f\le 
Q/e\\ (f,2k)=1}} \mu^2(f) \left(\frac{-a_2\cdots a_rb}{fg}\right)\times\\ & \sum\limits_{h=1}^8 \left\{\frac{-ha_2\cdots a_rb}{2^{v(ke)}}\right\} 
\sum\limits_{\substack{a_1\in \mathcal{A}_1(ef)\\ \congr{a_1}{h}{8}\\
(a_1,ke)=1}}  \Phi(a_1,...,a_r,ef) \left(\frac{a_1}{fg}\right),
\end{split}
\end{equation}
where $\mathcal{A}_1(ef)$ is defined in \eqref{Adef} and \eqref{Anotation}.
In the following, we estimate the inner-most sum over $a_1$ under the
assumption $\mu^2(fg)=1$. Using partial summation and the assumptions on
$\Phi$ in Section 2.1 (in particular, \eqref{Vcond}), we get
\begin{equation} \label{partsum}
\sum\limits_{\substack{a_1\in \mathcal{A}_1(ef)\\ \congr{a_1}{h}{8}\\
(a_1,ke)=1}}  \Phi(a_1,...,a_r,ef) \left(\frac{a_1}{fg}\right) \ll V \cdot \sup\limits_{L_1<L_2} \left| \sum\limits_{\substack{L_1<a_1\le L_2\\ a_1\in \mathcal{A}_1(ef)\\ \congr{a_1}{h}{8}\\
(a_1,ke)=1}}  \left(\frac{a_1}{fg}\right)\right|.
\end{equation}
Removing the coprimality condition $(a_1,ke)=1$ using a M\"obius inversion, we obtain
\begin{equation} \label{mobi}
\sum\limits_{\substack{L_1<a_1\le L_2\\ a_1\in \mathcal{A}_1(ef)\\  \congr{a_1}{h}{8}\\
(a_1,ke)=1}}  \left(\frac{a_1}{fg}\right)= \sum\limits_{d|ke} \mu(d) \left(\frac{d}{fg}\right) 
\sum\limits_{\substack{L_1/d<a\le L_2/d\\ da\in \mathcal{A}_1(ef)\\ \congr{da}{h}{8}\\}}  \left(\frac{a}{fg}\right).
\end{equation}
Recalling the assumption that $\mathcal{A}_1(ef)$ is the union of
finitely many intervals whose number is bounded by an absolute constant, the Polya-Vinogradov bound for character sums gives
\begin{equation} \label{polvi}
\sum\limits_{\substack{L_1/d<a\le L_2/d\\ da\in \mathcal{A}_1(ef)\\ \congr{da}{h}{8}\\}}  \left(\frac{a}{fg}\right) \ll f^{1/2} g^{1/2}\log(fg),
\end{equation}
where we note that $fg$ is not a perfect square since $fg>1$ and $\mu^2(fg)=1$.
Combining \eqref{Seg}, \eqref{partsum}, \eqref{mobi} and \eqref{polvi}, we get
\begin{equation*}
E(e,g)\ll VK_2\cdots K_rQ^{3/2}e^{-3/2}g^{1/2}\log(2gQe^{-1}) 2^{\omega(ke)}.
\end{equation*}
Similarly, for every $i\in \{1,...,r\}$, we obtain
\begin{equation*} 
E(e,g)\ll V\cdot \frac{K_1\cdots K_r}{K_i}\cdot Q^{3/2}e^{-3/2} g^{1/2}\log(2gQe^{-1}) 2^{\omega(ke)}.
\end{equation*}
Hence, on taking $K_i$ as the maximum of $K_1,...,K_r$, it follows that
\begin{equation} \label{Sesti1}
E(e,g)\ll VK^{1-1/r}Q^{3/2}e^{-3/2} g^{1/2}\log(2gQe^{-1}) 2^{\omega(ke)},
\end{equation}
where $K$ is defined in \eqref{Kdef}.

\subsection{Application of the Polya-Vinogradov bound II} \label{PolvinII}
In this section, we set $a = a_1\cdots a_r$.  Alternatively, we may use the Polya-Vinogradov bound to treat the
inner-most sum over $f$ in \eqref{Segdef} non-trivially if $-ab$ is not a perfect square, which we assume in the following. Using partial summation and the bound \eqref{Vcond}, we deduce
\begin{equation} \label{partsum2}
\begin{split}
& \sum\limits_{\substack{Q^-(e,g) < f \le  Q^{+}(e) \\ (f,2k)=1}} \Phi(a_1,...,a_r,ef) \mu^2(f)\left(\frac{-ab}{fg}\right) \\
\ll {} & V \cdot \sup\limits_{Q^-(e,g)\le F_1<F_2\le Q^{+}(e)} \left| \sum\limits_{\substack{F_1<f\le F_2\\ (f,2k)=1}} \mu^2(f)\left(\frac{-ab}{f} \right)
 \right|.
\end{split}
\end{equation}
Using the well-known formula
\begin{equation*}
\mu^2(f)=\sum\limits_{d^2|f} \mu(d),
\end{equation*}
and writing $f=d^2\tilde{f}$, we get
\begin{equation} \label{remove}
\sum\limits_{\substack{F_1<f\le F_2\\ (f,2k)=1}} \mu^2(f)\left( \frac{-ab}{f} \right) = \sum\limits_{\substack{d\le F_2^{1/2}\\ (d,2abk)=1}} \mu(d)
\sum\limits_{\substack{F_1/d^2<\tilde{f}\le F_2/d^2\\ (\tilde{f},2k)=1}} \left( \frac{-ab}{\tilde{f}} \right).
\end{equation}
Removing the coprimality condition $(\tilde{f},k)=1$ using a M\"obius inversion, we obtain
\begin{equation} \label{mobi2}
\sum\limits_{\substack{F_1/d^2<\tilde{f}\le F_2/d^2\\ (\tilde{f},2k)=1}} \left( \frac{-ab}{\tilde{f}} \right)
= \sum\limits_{\substack{\tilde{d}|k\\ (\tilde{d},2)=1}} \mu(\tilde{d}) \left( \frac{-ab}{\tilde{d}} \right) \sum\limits_{\substack{F_1/(d^2\tilde{d})<f'\le F_2/(d^2\tilde{d})\\ (f',2)=1}} \left( \frac{-ab}{f'} \right).
\end{equation}
The Polya-Vinogradov bound gives
\begin{equation} \label{polvi2}
\sum\limits_{\substack{F_1/(d^2\tilde{d})<f'\le F_2/(d^2\tilde{d})\\ (f',2)=1}} \left( \frac{-ab}{f'} \right) \ll \left(a|b|\right)^{1/2} \log(2a|b|),
\end{equation}
where we recall our assumption that $-ab$ is not a perfect square.

Let $E'(e,g)$ be the contribution to $E(e,g)$ of those $a_1,...,a_r$  for which $-ab$ is not a perfect square. Then, combining \eqref{Qcond}, \eqref{partsum2}, \eqref{remove}, \eqref{mobi2} and \eqref{polvi2}, we get
\begin{equation} \label{Sestiprime}
E'(e,g) \ll VK^{3/2} Q^{1/2} e^{-1/2}|b|^{1/2}\log(K|b|) 2^{\omega(k)}.
\end{equation}
The remaining contribution $E^{\Box}(e,g)$ of perfect squares $-ab$ is trivially calculated to be 
\begin{equation} \label{remain}
E^{\Box}(e,g)\ll VK^{1/2+\varepsilon} Qe^{-1}.
\end{equation}
Combining \eqref{Sestiprime} and \eqref{remain}, we obtain
\begin{equation} \label{Sesti2}
E(e,g)\ll VK^{3/2} Q^{1/2} e^{-1/2}|b|^{1/2}\log(K|b|) 2^{\omega(k)}+VK^{1/2+\varepsilon} Qe^{-1}.
\end{equation}

\subsection{Application of Heath-Brown's large sieve} \label{Heabro}
Finally, we may make use of Heath-Brown's large sieve for real character sums to bound $E(e,g)$, which we shall do in the following. Let us set
\begin{equation*}
u_f=\Phi(a_1,...,a_r,ef) \mu^2(f)\left(\frac{-a_1\cdots a_rb}{fg}\right).
\end{equation*}
To make the summation ranges independent, we first remove the
summation condition $Q^-(e,g)<f\le Q^{+}(e)$ using Perron's formula, getting
\begin{equation} \label{perronrem}
\begin{split}
& \sum\limits_{\substack{Q^-(e,g)<f\le 
Q^{+}(e)\\ (f,2k)=1}} u_f = \\
& \frac{1}{2\pi i} \int\limits_{c-iT}^{c+iT} \left(\sum\limits_{\substack{1\le f \le Q/e\\ (f,2k)=1}} u_f f^{-s}\right)
\left(Q^{+}(e)^{s}-Q^-(e,g)^{s} \right) \frac{\mbox{\rm d} s}{s} + O\left(V+\frac{VQ\log 2Q}{eT}\right),
\end{split}
\end{equation}
where we have set $c=1/\log 2Q$ and we have used \eqref{Vcond}. Set
\begin{equation*}
T=2Q(\log 2Q)e^{-1},
\end{equation*}
\begin{equation*}
A(a_1,...,a_r;s)=\left(Q^{+}(e)^{s}-Q^-(e,g)^{s} \right)
\left\{\frac{-a_1\cdots a_rb}{2^{v(ke)}}\right\}\left(\frac{-a_1\cdots a_rb}{g}\right),
\end{equation*}
\begin{equation*}
B(f;s)=f^{-s} \mu^2(f)\left(\frac{-b}{f}\right).
\end{equation*}
and 
\begin{equation*} \label{Isdef}
\begin{split}
I(s)= \mathop{\sum\limits_{a_1\sim K_1}  \cdots \sum\limits_{a_r \sim K_r}}_{(a_1\cdots a_r,ke)=1} \sum\limits_{\substack{1\le f \le Q/e\\ (f,2)=1}}\Phi(a_1,...,a_r,ef) A(a_1,...,a_r;s) B(f;s)
\left(\frac{a_1\cdots a_r}{f}\right).
\end{split}
\end{equation*}
Then it follows from \eqref{perronrem} that
\begin{equation} \label{Seg3}
\begin{split}
E(e,g) & = \frac{1}{2\pi i} \int\limits_{c-iT}^{c+iT} I(s) \frac{\mbox{\rm d} s}{s} +O\left(VK\right)
\ll  \left(\log T\right) \sup\limits_{-T\le t \le T} |I(c+it)| +VK \\ & =   \left(\log T\right) |I(c+it_0)| +VK
\end{split}
\end{equation}  
for a particular $t_0\in [-T,T]$.  From
\cite[Corollary~4]{MR1347489}, a version of Heath-Brown's large sieve
for real character sums, we have
\begin{equation} \label{ls}
\sum\limits_{a_1\sim K_1}  \cdots \sum\limits_{a_r \sim K_r} \
\sum\limits_{\substack{1\le f \le F\\ (f,2)=1}}A'(a_1,...,a_r)B'(f) \left(\frac{a_1\cdots a_r}{f}\right) 
\ll \left(KF^{1/2}+K^{1/2}F\right)\left(KF\right)^{\varepsilon}
\end{equation}
whenever $A'(a_1,...,a_r),B'(f)\ll 1$ and $F\ge 1$, and where we note that
\begin{equation*}
\left| \mathop{\sum\limits_{a_1\sim K_1}  \cdots \sum\limits_{a_r \sim K_r}}_{a_1\cdots a_r=a}  A'(a_1,...,a_r) \right| \ll \tau_r(a) \ll a^{\varepsilon} 
\end{equation*}
for any given $a\in \ZZp$ and where $\tau_r$ denotes the Dirichlet convolution of the constant arithmetic function equal to $1$ by itself $r$ times.  Using the bound \eqref{ls} together with partial summation in $f$ to remove the weight function  $\Phi(a_1,...,a_r,ef)$, we deduce that
\begin{equation} \label{Seg4}
|I(c+it_0)| \ll  V\left(KQ^{1/2}e^{-1/2}+K^{1/2}Qe^{-1}\right)\left(KQe^{-1}\right)^{\varepsilon},
\end{equation}
where we take into account that 
\begin{equation*}
A(a_1,...,a_r;t_0) \ll 1,
\end{equation*}
\begin{equation*}
B(f;t_0) \ll 1.
\end{equation*}
Combining \eqref{Seg3} and \eqref{Seg4}, and noting that
\begin{equation*}
\log T = \log \frac{2Q\log 2Q}{e} = \log\left(\frac{2Q}{e}\right) + \log\log (2Q) \ll \left(\frac{Q}{e}\right)^{\varepsilon} \log^{\varepsilon}(2+Q),
\end{equation*}
we deduce that
\begin{equation} \label{Sesti4}
E(e,g) \ll  V\left(KQ^{1/2}e^{-1/2}+K^{1/2}Qe^{-1}\right)\left(KQe^{-1}\right)^{\varepsilon}  \log^{\varepsilon}(2+Q).
\end{equation}

\subsection{Proofs of Theorem \ref{secondresult} and Corollary \ref{secondresultcorollary}}

We start by proving Theorem~\ref{secondresult}.

\begin{proof}
Combining the three bounds \eqref{Sesti1}, \eqref{Sesti2} and \eqref{Sesti4}, we obtain
\begin{equation} \label{Seg5}
E(e,g)\ll \left(V \left( KQe^{-1} \right)^{\varepsilon}\log^{\varepsilon}(2+Q)\right)\mathbf{m}+VK^{1/2+\varepsilon} Qe^{-1},
\end{equation}
where 
\begin{equation*}
\begin{split}
\mathbf{m}={} & \min\Big\{K^{1-1/r}Q^{3/2}e^{-3/2} g^{1/2+\varepsilon}, K^{3/2} Q^{1/2} e^{-1/2}|b|^{1/2+\varepsilon}2^{\omega(k)},\\ & \quad\quad\quad KQ^{1/2}e^{-1/2}+K^{1/2}Qe^{-1}\Big\}\\
\ll{}&  \min\left\{K^{1-1/r}Q^{3/2}e^{-3/2} g^{1/2+\varepsilon}, KQ^{1/2}e^{-1/2} \right\} +\\ &
\min\left\{K^{3/2} Q^{1/2} e^{-1/2}|b|^{1/2+\varepsilon}2^{\omega(k)}, K^{1/2}Qe^{-1}\right\}\\
\ll{} & \left(K^{1-1/r}Q^{3/2}e^{-3/2} g^{1/2+\varepsilon}\right)^{\mu}
\left(KQ^{1/2}e^{-1/2}\right)^{1-\mu}+\\ &
\left(K^{3/2} Q^{1/2} e^{-1/2}|b|^{1/2+\varepsilon}2^{\omega(k)}\right)^{\nu} \left(K^{1/2}Qe^{-1}\right)^{1-\nu}\\
\ll{} & K^{1-\mu/r}Q^{1/2+\mu}e^{-(1/2+\mu)}g^{\mu/2+\varepsilon}+
K^{1/2+\nu}Q^{1-\nu/2}e^{-(1-\nu/2)}|b|^{\nu/2+\varepsilon}2^{\nu \omega(k)}
\end{split}
\end{equation*}
for any $\mu,\nu\in [0,1]$. Choosing $(\mu,\nu)=(1/2-3\varepsilon, 4 \varepsilon)$, recalling \eqref{errorrewritten} and \eqref{Seg5}, and summing over $g$ and $e$ now gives
\begin{equation*} 
\begin{split}
E\ll {}&VK^{1-1/(2r)+\varepsilon}Q\rad(k)^{1/4}\log^{\varepsilon}(2+Q)\\ & 
+VK^{1/2+4 \varepsilon}Q|b|^{3\varepsilon}2^{(1+4\varepsilon)\omega(k)}\log^{\varepsilon}(2+Q)+
VK^{1/2+\varepsilon}Q\log(2+Q)2^{\omega(k)},
\end{split}
\end{equation*}
which ends the proof of Theorem \ref{secondresult}. 
\end{proof}

We can now deduce Corollary \ref{secondresultcorollary}. 

\begin{proof}
Removing all additional coprimality conditions separately using M\"obius inversions, i.e. the formula
\begin{equation*}
\sum\limits_{d|(m,n)} \mu(d)=\begin{cases} 1 & \mbox{ if } (m,n)=1\\  0 & \mbox{ otherwise,} \end{cases}
\end{equation*}
we are led to
\begin{equation} \label{lisa}
\begin{split}
  \Sigma'={}&\sum\limits_{\substack{(d_{\alpha,\beta})\in \ZZp^{r(r-1)/2}\\(1 \le \alpha < \beta \le r)}}
  \ \sum\limits_{d_1|t_1} \cdots
  \sum\limits_{d_r|t_r} \sum\limits_{d|u} \left(\prod\limits_{1\le
      i<j\le r} \mu(d_{i,j}) \right) \left(\prod\limits_{l=1}^r
    \mu(d_{l})\right) \mu(d) \times\\ & \Sigma((d_{i,j})_{1\le i<j\le
    r},d_1,...,d_r,d)
\end{split}
\end{equation}
with
\begin{equation} \label{homer}
\begin{split}
& \Sigma((d_{i,j})_{1\le i<j\le r},d_1,...,d_r,d)= \\ & \sum\limits_{a_1\sim K_1/D_1} \cdots \sum\limits_{a_r \sim K_r/D_r}\ \sum\limits_{Q^{-}/d<q\le 
Q^{+}/d}  \Phi(a_1D_1,...,a_rD_r,qd) \cdot \N(-aDb,kdq),
\end{split}
\end{equation}
where
\begin{equation*}
a= a_1\cdots a_r,
\end{equation*}
\begin{equation*}
D_i=\mbox{lcm}(d_i, d_{1,i},...,d_{i-1,i},d_{i,i+1},...,d_{i,r})
\end{equation*}
and 
\begin{equation*}
D=D_1\cdots D_r.
\end{equation*}
Using Theorem \ref{secondresult}, we obtain
\begin{equation} \label{marge}
\begin{split}
& \Sigma((d_{i,j})_{1\le i<j\le r},d_1,...,d_r,d) - M((d_{i,j})_{1\le i<j\le r},d_1,...,d_r,d) \ll \\ &
V\left(\frac{K}{D}\right)^{1/2+\varepsilon}\frac{Q}{d} L^{\varepsilon} \left( \left(\frac{K}{D}\right)^{1/2-1/2r} d^{1/4} \rad(k)^{1/4}+
|Db|^{\varepsilon} 2^{(1+\varepsilon)\omega(dk)}+2^{\omega(dk)}L \right),
\end{split}
\end{equation}
where $L=\log(2+Q)$ and 
\begin{equation*}
\begin{split}
& M((d_{i,j})_{1\le i<j\le r},d_1,...,d_r,d)= \\ & \sum\limits_{a_1\sim K_1/D_1}   \cdots \sum\limits_{a_r \sim K_r/D_r}\ \sum\limits_{Q^{-}/d<q\le 
Q^{+}/d}  \Phi(a_1D_1,...,a_rD_r,qd) \cdot \left\{\frac{-a_1\cdots a_rDb}{2^{v(kdq)}}\right\}.
\end{split}
\end{equation*}
Reverting all the M\"obius inversions carried out, we find that
\begin{equation*}
\begin{split}
M'={}& \sum\limits_{\substack{(d_{\alpha,\beta})\in \ZZp^{r(r-1)/2}\\(1 \le \alpha < \beta \le r)}}
\ \sum\limits_{d_1|t_1} \cdots \sum\limits_{d_r|t_r} \sum\limits_{d|u} \left(\prod\limits_{1\le i<j\le r} \mu(d_{i,j}) \right) 
\left(\prod\limits_{l=1}^r \mu(d_{l})\right) \mu(d) \times\\ & M((d_{i,j})_{1\le i<j\le r},d_1,...,d_r,d),
\end{split}
\end{equation*}
where $M'$ is defined in \eqref{main'}. Summing
up the error term in \eqref{marge} over $D\le K$ and $d\le Q^{-}$, and noting that the number of $d_{\alpha,\beta}$'s and $d_{\gamma}$'s such that
\begin{equation*}
D=D_1\cdots D_r= \prod\limits_{i=1}^r \mbox{lcm}(d_i, d_{1,i},...,d_{i-1,i},d_{i,i+1},...,d_{i,r})
\end{equation*}
is bounded by $O\left(D^{\varepsilon}\right)$, we get the error term claimed in Corollary~\ref{secondresultcorollary},
which ends the proof.
\end{proof}

\section{Counting rational points on a singular cubic surface}

In this part, we give a proof of Manin's conjecture (Theorem~\ref{thm:main})
for the singular cubic surface with $\Afive+\Aone$ singularity type. We will
apply our result on quadratic congruences (Corollary~\ref{secondresultcorollary}).

\subsection{Geometry}\label{sec:geometry}

Our cubic surface $S$ defined by~(\ref{eq:surface}) over the field $\QQ$ has
singularities only in $(0:0:1:0)$, of type $\Aone$, and $(1:0:0:0)$ of type
$\Afive$. It contains precisely two lines $\{x_1=x_2=0\}$ and
$\{x_1=x_3=0\}$. The complement of the lines is $U = \{\xx \in S \mid x_1\ne
0\}$.  It is rational, as one can see by projecting to $\Ptwo$ from one of the
singularities.

Its minimal desingularization $\tS$ is a blow-up of $\Ptwo$ in six points,
so $\Pic(\tS)$ is free of rank $7$. The Cox ring of $\tS$ has been determined in
\cite{MR3180592}. It has $10$ generators $\e_1, \dots, \e_{10}$
satisfying the relation~(\ref{eq:torsor1}). The configuration of the rational
curves on $\tS$ corresponding to the generators of $\Cox(\tS)$ is described by
the extended Dynkin diagram in Figure~\ref{fig:dynkin}, where each vertex
corresponds to a curve $E_i$ of $\e_i$, and an edge indicates that two curves
intersect.

\begin{figure}[ht]
  \centering
    \[\xymatrix@R=0.05in @C=0.05in{E_9 \ar@{-}[rrrrr] \ar@{-}[dr] \ar@{-}[dd]& & & & & \ex{E_2} \ar@{-}[dr]\\
      & \ex{E_8} \ar@{-}[r] & E_6 \ar@{-}[r] & \ex{E_7} \ar@{-}[r] & \ex{E_5} \ar@{-}[r] & \ex{E_4} \ar@{-}[r] & \ex{E_3}\\
      E_{10} \ar@{-}[rrrrr] \ar@{-}[ur] & & & & & E_1 \ar@{-}[ur]}\]
  \caption{Configuration of curves on $\tS$.}
  \label{fig:dynkin}
\end{figure}
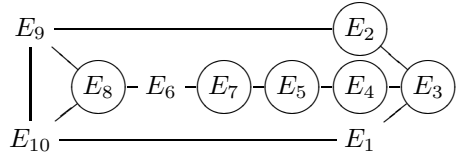

\subsection{Passage to a universal torsor}\label{sec:torsor}

Let
\begin{equation*}
  \ee = (\e_1, \dots, \e_{10}), \quad \ee' = (\e_1, \dots, \e_8),\quad \ee^{(k_1, \dots, k_8)} = \e_1^{k_1}\cdots \e_8^{k_8},
\end{equation*}
for any $(k_1, \dots, k_8) \in \RR^8$.

For $i=1, \dots, 10$, let
\begin{equation}\label{eq:intervals}
  (\ZZ_i, J_i, J_i') =
  \begin{cases}
    (\ZZp, \RR_{\ge 1}, \RR_{\ge 1}), & i \in \{1, \dots, 6\},\\
    (\ZZp, \RR_{\ge 1}, \RRnn), & i = 7,\\
    (\ZZnz, \RR_{\le -1} \cup \RR_{\ge 1}, \RR), & i = 8,\\
    (\ZZ, \RR, \RR), & i \in \{9, 10\}.
  \end{cases}
\end{equation}
In the course of our argument, we estimate summations over $\e_i \in
\ZZ_i$ by integrations over $\e_i \in J_i$, which we enlarge to $\e_i
\in J_i'$ in (\ref{eq:new_region}).

\begin{lemma}\label{lem:bijection}
  We have
  \begin{equation*}
    N_{U,H}(B)= \# \{\ee \in \ZZ_1 \times \dots \times \ZZ_{10} \mid \text{(\ref{eq:torsor})--(\ref{eq:cpe}) hold}\}
  \end{equation*}
  with the torsor equation
  \begin{equation}\label{eq:torsor}
    \e_1\e_{10} + \e_2\e_9^2+\e_4\e_5^2\e_6^4\e_7^3\e_8 = 0,
  \end{equation}
  the height condition
  \begin{equation}\label{eq:height}
    h(\ee', \e_9;B) = B^{-1}\max\left\{
    \begin{aligned}
      &|\e_1^{-1}(\e_2\e_8\e_9^2+\e_4\e_5^2\e_6^4\e_7^3\e_8^2)|,|\base 1 1 2
      2 2 2 2 1|,\\
      &|\base 3 2 4 3 2 0 1 0|, |\base 0 1 1 1 1 1 1 1 \e_9|
    \end{aligned}
    \right\}\le 1
  \end{equation}
  and the coprimality conditions
  \begin{align}
    \label{eq:cpe910} &(\e_{10},\e_2\e_3\e_4\e_5\e_6\e_7)=(\e_9,\e_1\e_3\e_4\e_5\e_6\e_7)=1,\\
    \label{eq:cpe8} &\cp{\e_8}{\e_1\e_2\e_3\e_4\e_5\e_7}, \\
    \label{eq:cpe}
    &\begin{aligned}
      (\e_7,\e_1\e_2\e_3\e_4)&=(\e_6,\e_1\e_2\e_3\e_4\e_5)=(\e_5,\e_1\e_2\e_3)\\
      &=(\e_4,\e_1\e_2)=(\e_1,\e_2)=1.
    \end{aligned}
  \end{align}
\end{lemma}

\begin{proof}
  Based on the birational projection $S \rto \Ptwo$ from the
  $\Afive$-singularity and the structure of $\tS$ as a blow-up of
  $\Ptwo$ in six points, we prove as in \cite[Section~4]{MR2290499}
  that the map
  \begin{equation*}
    \psi: \ee \mapsto (\e_8\e_{10}, \base 1 1 2 2 2 2 2 1, \base 3
    2 4 3 2 0 1 0 , \base 0 1 1 1 1 1 1 1 \e_9),
  \end{equation*}
  gives a bijection between the rational points on $U$ and the set of
  $\ee \in \ZZ_1 \times \dots \times \ZZ_{10}$
  satisfying~(\ref{eq:torsor}) and the coprimality conditions encoded in the
  extended Dynkin diagram in Figure~\ref{fig:dynkin}, which are
  (\ref{eq:cpe910})--(\ref{eq:cpe}).

  We note that the coprimality conditions imply that the image of such
  $\ee$ under $\psi$ has coprime coordinates, so that the height of
  $\psi(\ee)$ is simply the maximum of their absolute
  values. Using~(\ref{eq:torsor}), we eliminate $\e_{10}$ and
  obtain~(\ref{eq:height}).
\end{proof}

\subsection{Counting points}

Recalling the definition~(\ref{eq:intervals}) of $J_i$, let
\begin{equation*}
  \R(B) = \{(\ee', \e_9) \in J_1 \times \dots \times J_9 \mid h(\ee', \e_9;B) \le 1\}
\end{equation*}
be the set whose number of lattice points we want to compare with its volume
(both under the torsor equation (\ref{eq:torsor}) and the coprimality
conditions~(\ref{eq:cpe910})--(\ref{eq:cpe})).

Recall the definition~(\ref{eq:rho-term}) of $\N(q,a)$.  Summing over $\e_9$,
with $\e_{10}$ as a dependent variable, we get:

\begin{lemma}\label{lem:first_sum}
  We have
  \begin{equation*}
    N_{U,H}(B) = \sum_{\ee' \in \ZZ_1 \times \dots \times \ZZ_8} \theta_1(\ee')V_1(\ee';B) + O(B(\log B)^3),
  \end{equation*}
  where
  \begin{equation}\label{eq:integral}
    V_1(\ee';B) = \int_{(\ee', \e_9) \in \R(B)} \e_1^{-1} \dd \e_9
  \end{equation}
  and
  \begin{equation*}
    \theta_1(\ee') = \sum_{\substack{k \mid \e_3 \\ \cp{k}{\e_2\e_4}}}
    \frac{\mu(k)\varphi^*(\e_3\e_4\e_5\e_6\e_7)}{k\varphi^*((\e_3,k\e_1))}
    \N(-\e_2\e_4\e_7\e_8,k\e_1)
  \end{equation*}
  if $\ee'$ satisfies the coprimality
  conditions~(\ref{eq:cpe8})--(\ref{eq:cpe}), while $\theta_1(\ee')=0$
  otherwise.
\end{lemma}

\begin{proof}
  Essentially because Figure~\ref{fig:dynkin} describing the coprimality
  conditions and the torsor equation~(\ref{eq:torsor}) have the right shape,
  we are in the position to apply the general result of
  \cite[Proposition~2.4]{MR2520770}. This gives the main term as above after
  we simplify the condition $\cp{k}{\e_2\e_4\e_5\e_6\e_7\e_8}$ in the
  summation over $k$ to $\cp{k}{\e_2\e_4}$, which is allowed because of $k
  \mid \e_3$ and (\ref{eq:cpe8})--(\ref{eq:cpe}).

  The sum of the error term over all relevant $\ee'$ is bounded by
  \begin{equation*}
    \begin{split}
      \sum_{\ee'}
      2^{\omega(\e_3)+\omega(\e_3\e_4\e_5\e_6\e_7)+\omega(\e_1\e_3)}
      &\ll \sum_{\e_1, \dots, \e_7}
      \frac{2^{\omega(\e_3)+\omega(\e_3\e_4\e_5\e_6\e_7)+\omega(\e_1\e_3)}B}{\base
        1 1 2 2 2 2 2 0}\\
      & \ll B(\log B)^3,
    \end{split}
  \end{equation*}
  where we use the second part of~(\ref{eq:height}) for the summation over
  $\e_8$.
\end{proof}

\subsection{Application of Corollary \ref{secondresultcorollary}}\label{sec:application}
Using Corollary \ref{secondresultcorollary}, we now want to prove that
Lemma~\ref{lem:first_sum} still holds when we replace the error term by
$O(B(\log B)^{4+\varepsilon})$ and $\theta_1$ in the main term by $\theta_1'$
with
\begin{equation*}
  \theta_1'(\ee') = \sum_{\substack{k \mid \e_3 \\ \cp{k}{\e_2\e_4}}}
  \frac{\mu(k)\varphi^*(\e_3\e_4\e_5\e_6\e_7)}{k\varphi^*((\e_3,k\e_1))} \left\{\frac{-\e_2\e_4\e_7\e_8}{2^{v(k\e_1)}}\right\}
\end{equation*}
if (\ref{eq:cpe8})--(\ref{eq:cpe}) hold and $\theta_1'(\ee')=0$
otherwise. Hence, we want to show the following.
\begin{lemma}\label{lem:first_sum2}
Let $\ep>0$ be fixed. We have
  \begin{equation*}
    N_{U,H}(B) = \sum_{\ee' \in \ZZ_1 \times \dots \times \ZZ_8} \theta_1'(\ee')V_1(\ee';B) + O(B(\log B)^{4+\varepsilon}).
  \end{equation*}
\end{lemma}
  
\begin{proof} 
First, we write
\begin{equation*}
\sum\limits_{\ee'\in \ZZ_1\times\cdots \times \ZZ_8} \theta_1(\ee')V_1(\ee'; B)=F^+(B)+F^-(B),
\end{equation*}
where
\begin{equation*}
F^+(B)=\sum\limits_{\ee'\in \ZZp^7\times \ZZp}
\theta_1(\ee')V_1(\ee'; B),
\end{equation*}
and
\begin{equation*}
\ F^-(B)=\sum\limits_{\ee'\in \ZZp^7\times \ZZ_{<0}} \theta_1(\ee')V_1(\ee'; B).
\end{equation*}
The term $F^-(B)$ can be treated similarly as $F^+(B)$. Therefore, we confine
ourselves to the treatment of the term $F^+(B)$, which we now transform in
such a way that Corollary \ref{secondresultcorollary} can be applied.
 
For convenience, we break the summation ranges of $\eta_1, \eta_2, \eta_4, \eta_7$ and $\eta_8$
into dyadic intervals, i.e., we write
\begin{equation} \label{F+}
F^+(B) = \sum\limits_{\ee''\in \ZZp^3} \sum\limits_{k|\eta_3} \frac{\mu(k)}{k} \sum\limits_{L_1,L_2,L_4,L_7,L_8} W(\ee'',k,L_1,L_2,L_4,L_7,L_8),
\end{equation}
where $\ee''=(\eta_3,\eta_5,\eta_6)$ satisfies the coprimality conditions $(\eta_3,\eta_5\eta_6)=1=(\eta_5,\eta_6)$, the variables $L_1,L_2,L_4, L_7,L_8\ge 1/2$ run over powers of $2$, respectively, and
\begin{equation*}
\begin{split}
& W(\ee'',k,L_1,L_2,L_4,L_7,L_8) = \sum\limits_{\substack{\eta_1\sim L_1 \\ (\eta_1,\eta_5\eta_6)=1}} \varphi^{\ast}((\eta_3,k\eta_1))^{-1} \sum\limits_{\substack{\eta_4\sim L_4\\ (\eta_4,\eta_6)=1}} \sum\limits_{\substack{\eta_7\sim L_7\\ (\eta_7,\eta_3\eta_4)=1}} \\ & \varphi^{\ast}(\eta_3\eta_4\eta_5\eta_6\eta_7) \sum\limits_{\substack{\eta_2\sim L_2 \\ (\eta_2,\eta_4\eta_5\eta_6\eta_7)=1}} \sum\limits_{\substack{\eta_8\sim L_8 \\ (\eta_8,\eta_2\eta_3\eta_4\eta_5\eta_7)=1}} V_1(\ee; B) \N(-\e_2\e_4\e_7\e_8,k\e_1).
\end{split}
\end{equation*}
Here we note that the coprimality condition $(\eta_2\eta_4\eta_7\eta_8,k\eta_1)=1$ is contained in the definition of $\N(-\e_2\e_4\e_7\e_8,k\e_1)$.

To make Corollary \ref{secondresultcorollary} applicable, it is necessary to
remove the arithmetic factors $\varphi^{\ast}((\eta_3,k\eta_1))^{-1}$ and
$\varphi^{\ast}(\eta_3\eta_4\eta_5\eta_6\eta_7)$. We write
\begin{equation} \label{arithfac1}
\begin{split}
\varphi^{\ast}((\eta_3,k\eta_1))^{-1}&= \varphi^{\ast}(k\cdot (\eta_3/k,\eta_1))^{-1}=
\varphi^{\ast}(k)^{-1}\prod\limits_{\substack{p| (\eta_3/k,\eta_1)\\ p\nmid k}} \left(1+\frac{1}{p-1}\right)\\ 
&=  \varphi^{\ast}(k)^{-1}\sum\limits_{\substack{d_1| (\eta_3/k,\eta_1)\\ (d_1,k)=1}} 
\frac{\mu^2(d_1)}{\varphi(d_1)}
\end{split}
\end{equation}
and 
\begin{equation} \label{arithfac2}
\begin{split}
\varphi^{\ast}(\eta_3\eta_4\eta_5\eta_6\eta_7)&=  \varphi^{\ast}(\eta_3\eta_5\eta_6)
\prod\limits_{\substack{p| \eta_4\\ p\nmid \eta_3\eta_5\eta_6}} \left(1-\frac{1}{p}\right)
\prod\limits_{\substack{\tilde{p}| \eta_7\\ \tilde{p}\nmid \eta_3\eta_5\eta_6}} \left(1-\frac{1}{\tilde{p}}\right)\\ 
&=  \varphi^{\ast}(\eta_3\eta_5\eta_6)
\sum\limits_{\substack{d_4|\eta_4\\ (d_4,\eta_3\eta_5\eta_6)=1}} \frac{\mu(d_4)}{d_4}
\sum\limits_{\substack{d_7|\eta_7\\ (d_7,\eta_3\eta_5\eta_6)=1}} \frac{\mu(d_7)}{d_7},
\end{split}
\end{equation}
where we use the fact that $(\eta_4,\eta_7)=1$. Hence, we may write 
\begin{equation} \label{hence}
\begin{split}
 & W(\ee'',k,L_1,L_2,L_4,L_7,L_8)=\\ 
 & \frac{ \varphi^{\ast}(\eta_3\eta_5\eta_6)}{\varphi^{\ast}(k)} 
 \sum\limits_{\substack{d_1| \eta_3/k\\ (d_1,\eta_5\eta_6k)=1}} \sum\limits_{\substack{d_4\le 2L_4\\ (d_4,\eta_3\eta_5\eta_6)=1}} \sum\limits_{\substack{d_7\le 2L_7\\ (d_7,d_4\eta_3\eta_5\eta_6)=1}}\frac{\mu^2(d_1)\mu(d_4)\mu(d_7)}{\varphi(d_1)d_4d_7} \times \\ & W(\ee'',k,L_1,L_2,L_4,L_7,L_8,d_1,d_4,d_7),
\end{split}
\end{equation}
where 
\begin{equation*}
\begin{split}
 & W(\ee'',k,L_1,L_2,L_4,L_7,L_8,d_1,d_4,d_7)=\\ 
 & \sum\limits_{\substack{\eta_1'\sim L_1/d_1 \\ (\eta_1',\eta_5\eta_6)=1}} \sum\limits_{\substack{\eta_4'\sim L_4/d_4\\ (\eta_4',d_7\eta_6)=1}} \sum\limits_{\substack{\eta_7'\sim L_7/d_7\\ (\eta_7',d_4\eta_3\eta_4')=1}} \sum\limits_{\substack{\eta_2\sim L_2 \\ (\eta_2,d_4d_7\eta_4'\eta_5\eta_6\eta_7')=1}} \sum\limits_{\substack{\eta_8\sim L_8 \\
      (\eta_8,d_4d_7\eta_2\eta_3\eta_4'\eta_5\eta_7')=1}}\\ &  V_1(d_1\eta_1',\eta_2,\eta_3,d_4\eta_4',\eta_5,\eta_6,d_7\eta_7',\eta_8;
    B) 
  \N(-\e_2\e_4'\e_7'\e_8d_4d_7,kd_1\e_1').
  \end{split}
\end{equation*}

Now we observe that for $\eta_2,\eta_3,\eta_4',\eta_5,\eta_6,\eta_7',\eta_8>0$,  the set
\begin{equation*}
\left\{ y>0 \ |\  V_1(d_1y, \eta_2,\eta_3,d_4\eta_4',\eta_5,\eta_6,d_7\eta_7',\eta_8;B)> 0\right\}
\end{equation*}
is an interval. To evaluate $W(\ee'',k,L_1,L_2,L_4,L_7,L_8,d_1,d_4,d_7)$,
we shall apply Corollary \ref{secondresultcorollary} and Remark \ref{intervalrem} with 
\begin{equation*}
k \mbox{ replaced by } kd_1, \quad b=d_4d_7, \quad r=4,
\end{equation*}
\begin{equation*}
a_1=\eta_4',\quad
a_2=\eta_7',\quad a_3=\eta_2, \quad a_4=\eta_8, \quad q=\eta_1', 
\end{equation*}
\begin{equation*}
t_1=d_7\eta_6, \quad t_2=d_4\eta_3, \quad t_3=d_4d_7\eta_5\eta_6,\quad t_4=d_4d_7\eta_3\eta_5, \quad u=\eta_5\eta_6,
\end{equation*}
\begin{equation*}
K_1=L_4/d_4, \ \ K_2=L_7/d_7, \ \ K_3=L_2,\ \ K_4=L_8, \ \ Q=2L_1/d_1,
\end{equation*}
\begin{equation*}
\begin{split}
& \mathcal{I}(Q^-,Q^+)=\mathcal{I}\left(Q^-(\eta_4',\eta_7',\eta_2,\eta_8),Q^+(\eta_4',\eta_7',\eta_2,\eta_8)\right) = \\ 
& \left(L_1,2L_1\right]\cap \left\{ y>0 \mid  V_1(d_1y, \eta_2,\eta_3,d_4\eta_4',\eta_5,\eta_6,d_7\eta_7',\eta_8;B)> 0\right\},
\end{split}
\end{equation*}
\begin{equation*}
V=\sup_{\e_1 \sim L_1, \e_2 \sim L_2, \e_4 \sim L_4, \e_7\sim L_7, \e_8\sim L_8} V_1(\ee;B),
\end{equation*}
\begin{equation*}
\Phi(\eta_4',\eta_7',\eta_2,\eta_8,y)=
\begin{cases}
V_1(d_1y, \eta_2,\eta_3,d_4\eta_4',\eta_5,\eta_6,d_7\eta_7',\eta_8;B) & \mbox{ if } Q^-<y<Q^+, \\ 
\lim\limits_{z\downarrow Q^-} V_1(d_1z, \eta_2,\eta_3,d_4\eta_4',\eta_5,\eta_6,d_7\eta_7',\eta_8;B) & \mbox{ if } y\le Q^-, \\
\lim\limits_{z\uparrow Q^+} V_1(d_1z, \eta_2,\eta_3,d_4\eta_4',\eta_5,\eta_6,d_7\eta_7',\eta_8;B) & \mbox{ if } y\ge Q^+.
\end{cases}
\end{equation*}
It is easy to check that the so-defined functions $\Phi$, $Q^-$ and $Q^+$
satisfy the conditions in Section~\ref{counting}. Therefore, applying
Corollary \ref{secondresultcorollary} and Remark \ref{intervalrem} gives
\begin{equation} \label{splittingagain}
\begin{split}
W(\ee'',k,L_1,L_2,L_4,L_7,L_8,d_1,d_4,d_7)&=  M(\ee'',k,L_1,L_2,L_4,L_7,L_8,d_1,d_4,d_7)\\
&+E(\ee'',k,L_1,L_2,L_4,L_7,L_8,d_1,d_4,d_7),
\end{split}
\end{equation}
where 
\begin{equation} \label{frankfurt}
\begin{split}
  & M(\ee'',k,L_1,L_2,L_4,L_7,L_8,d_1,d_4,d_7)\\ =& \mathop{\sum\limits_{\substack{\eta_1'\sim L_1/d_1 \\ (\eta_1',\eta_5\eta_6)=1}} \sum\limits_{\substack{\eta_4'\sim L_4/d_4\\ (\eta_4',d_7\eta_6)=1}} \sum\limits_{\substack{\eta_7'\sim L_7/d_7\\ (\eta_7',d_4\eta_3\eta_4')=1}} \sum\limits_{\substack{\eta_2\sim L_2 \\ (\eta_2,d_4d_7\eta_4'\eta_5\eta_6\eta_7')=1}} \sum\limits_{\substack{\eta_8\sim L_8 \\
      (\eta_8,d_4d_7\eta_3\eta_4'\eta_5\eta_7')=1}}}_{(\eta_2\eta_4'\eta_7'\eta_8d_4d_7,kd_1\eta_1')=1} \\ &  V_1(d_1\eta_1',\eta_2,\eta_3,d_4\eta_4',\eta_5,\eta_6,d_7\eta_7',\eta_8;
    B) 
  \left\{\frac{\eta_2\eta_4'\eta_7'\eta_8d_4d_7}{2^{v(kd_1\eta_1')}}\right\}
\end{split}
\end{equation}
and
\begin{equation} \label{shanghai}
\begin{split}
  & E(\ee'',k,L_1,L_2,L_4,L_7,L_8,d_1,d_4,d_7)\\ \ll & 
  \sup_{\e_i \sim L_i} V_1(\ee;B) \cdot 
  \left(L_1(L_2L_4L_7L_8)^{7/8+\varepsilon}d_1^{-3/4}(d_4d_7)^{-7/8}k^{1/4}\right.\\ & \left. + 
    L_1(L_2L_4L_7L_8)^{1/2+4\varepsilon}d_1^{-1}(d_4d_7)^{-1/2}(\log4L_1)2^{(1+4\varepsilon)\omega(kd_1)}   \right)\times\\ & (1+\varepsilon)^{\omega(\e_3)+\omega(\e_5)+\omega(\e_6)}\log^{\varepsilon}(4L_1).
\end{split}
\end{equation}
Summing these contributions over $k$, $L_i$ and $d_i$, we deduce from \eqref{F+}, \eqref{hence}, \eqref{splittingagain}, \eqref{frankfurt} and \eqref{shanghai} that
\begin{equation} \label{F+break}
F^+(B)=M^+(B)+E^+(B),
\end{equation}
where 
\begin{equation*}
\begin{split}
M^+(B)={} & \sum\limits_{\ee''\in \ZZp^3}  \varphi^{\ast}(\eta_3\eta_5\eta_6)\sum\limits_{k|\eta_3} \frac{\mu(k)}{k\varphi^{\ast}(k)}  \sum\limits_{L_1,L_2,L_4,L_7,L_8} \sum\limits_{\substack{d_1| \eta_3/k\\ (d_1,\eta_5\eta_6k)=1}}
\sum\limits_{\substack{d_4\le 2L_4\\ (d_4,\eta_3\eta_5\eta_6)=1}} \\ &
\sum\limits_{\substack{d_7\le 2L_7\\ (d_7,d_4\eta_3\eta_5\eta_6)=1}}\frac{\mu^2(d_1)\mu(d_4)\mu(d_7)}{\varphi(d_1)d_4d_7} 
 M(\ee'',k,L_1,L_2,L_4,L_7,L_8,d_1,d_4,d_7),
\end{split}
\end{equation*}
and
\begin{equation} \label{E+}
E^+(B)\ll \sum\limits_{\ee''\in \ZZp^3}  (1+\varepsilon)^{\omega(\e_3)+\omega(\e_5)+\omega(\e_6)} \sum\limits_{L_1,L_2,L_4,L_7,L_8}
\LL \sup_{\e_i \sim L_i} V_1(\ee;B),
\end{equation}
where we have set
\begin{equation*}
\LL=L_1(L_2L_4L_7L_8)^{8/9}(\log4L_1)^{1+\varepsilon}.
\end{equation*}
Reverting the decompositions of the arithmetic functions in \eqref{arithfac1}
and \eqref{arithfac2}, combining the $\eta_1$-, $\eta_2$-, $\eta_4$-,
$\eta_7$- and $\eta_8$-ranges, and noting that if $k|\eta_3$ then the conditions
$(\eta_2\eta_4\eta_7\eta_8,k\eta_1)=1$ and $(k,\eta_2\eta_4)=1$ are equivalent,
we simplify the main term $M^+(B)$ into
\begin{equation} \label{M+}
M^+(B)=\sum\limits_{\ee'\in \ZZp^7\times \ZZp} \theta_1'(\ee')V_1(\ee'; B),
\end{equation}
where $\theta_1'(\ee')$ is defined in Lemma \ref{lem:first_sum2}.

Finally, we show that  $E^+(B)$  is an error term. To estimate $V_1$,
an application of \cite[Lemma~5.1]{MR2520770} gives
\begin{align}
  \label{eq:V1_bounda}V_1(\ee';B) & \ll \min \left\{
    \frac{B^{1/2}}{\e_1^{1/2}\e_2^{1/2}|\e_8|^{1/2}},
    \frac{B}{\base 0{1/2}0{1/2}12{3/2}{3/2}}\right\} \\
  \label{eq:V1_boundb}&\ll
  \frac{B^{2/3}}{|\base{1/3}{1/2}0{1/6}{1/3}{1/2}{2/3}{5/6}|} \\
  \label{eq:V1_bound}& = \frac{B}{|\base 1 1 1 1 1 1 1 1|} \left(\frac{B}{|\base 1 1 2 2 2 2
      2 1|}\frac{B}{|\base 3 2 4 3 2 0 1 0|}\right)^{-1/6},
\end{align}
where (\ref{eq:V1_boundb}) is the weighted average of the two parts
of~(\ref{eq:V1_bounda}), and~(\ref{eq:V1_bound}) indicates how the second and
third parts of the height condition~(\ref{eq:height}) will be used below when
summing over $\e_6,\e_7$.  Set
\begin{equation*}
\LL'=L_1(L_2L_4L_8)^{8/9}(\log4L_1)^{1+\varepsilon}.
\end{equation*}
Then, starting from \eqref{E+}, we see that
\begin{equation*}
  \begin{split}
    E^+(B) 
    &\ll  \sum_{L_1,L_2,L_4,L_7,L_8} \LL \sup_{\e_i \sim L_i} 
    \left(\sum_{\e_3, \e_5, \e_6}
      \frac{(1+\ep)^{\omega(\e_3)+\omega(\e_5)+\omega(\e_6)}B^{2/3}}{|\base{1/3}{1/2}0{1/6}{1/3}{1/2}{2/3}{5/6}|}\right)\\
      &\ll  \sum_{L_1,L_2,L_4,L_8}  \LL'
      \sup_{\e_i \sim L_i}
      \left(\sum_{\e_3, \e_5, \e_6,\e_7}
      \frac{(1+\ep)^{\omega(\e_3)+\omega(\e_5)+\omega(\e_6)}B^{2/3}}{|\base{1/3}{1/2}0{1/6}{1/3}{1/2}{2/3}{5/6}|}\right)\\
    &\ll\sum_{L_1,L_2,L_4,L_8} \LL' \sup_{\e_i \sim L_i}
    \left(\sum_{\e_3,\e_5} \frac{(1+\ep)^{\omega(\e_3)+\omega(\e_5)}B(\log
        B)^{\ep}}{|\e_1\e_2\e_3\e_4\e_5\e_8|}\right)\\
    &\ll\sum_{L_1,L_2,L_4,L_8} \LL' \sup_{\e_i \sim
      L_i} \frac{B(\log B)^{2+3\ep}}{|\e_1\e_2\eta_4\e_8|}\\
    &\ll\sum_{L_1,L_2,L_4,L_7} \frac{B(\log B)^{2+4\ep}(\log 4L_1)}{(L_2L_4L_8)^{1/9}}\\
    &\ll B(\log B)^{4+4\ep}.
  \end{split}
\end{equation*}
Combining this with \eqref{F+break} and \eqref{M+}, and treating $F^-(B)$
similarly as $F^+(B)$, we obtain the desired result.
\end{proof}

\subsection{Completion of the proof of Theorem \ref{thm:main}}\label{sec:completion}
For the proof of Theorem \ref{thm:main}, it remains to evaluate the main term
in Lemma \ref{lem:first_sum2} asymptotically. To this end, we would like to
apply \cite[Proposition~4.3]{MR2520770}. We note that $\theta_1'(\ee')$ is not
of the form considered in \cite[Section~7]{MR2520770} because of the extra
$2$-adic factor. However, this factor turns out to be $1$ on average, and the
remaining part of $\theta_1'(\ee')$ has the necessary properties. 
As in \cite[Definition~3.7]{MR2520770}, $\A(\theta_1'(\ee'),\e_8)$ denotes the
average size of $\theta_1'$ when summed over $\e_8$.

\begin{lemma}\label{lem:sum_e8}
  We have $\theta_1'(\ee') \in \Theta_{2,8}(C)$
  \cite[Definition~4.2]{MR2520770} for some $C \in \RRnn$, with
  \begin{equation*}
    \A(\theta_1'(\ee'),\e_8)
    = \theta_2(\e_1, \dots, \e_7) =
    \prod_p \theta_{2,p}(I_p(\e_1, \dots, \e_7)) \in \Theta_{4,7}'(2),
  \end{equation*}
  \cite[Definition~7.8]{MR2520770}, where $I_p(\e_1, \dots, \e_7) = \{i \in
  \{1, \dots, 7\} \mid p \mid \e_i\}$ and
  \begin{equation*}
    \theta_{2,p}(I) =
    \begin{cases}
      1, & I=\emptyset,\\
      1-\frac 1 p, & I=\{1\}, \{2\}, \{6\},\\
      (1-\frac 1 p)^2, &
      \begin{aligned}
        I=&\{4\}, \{5\}, \{7\}, \{1,3\}, \{2,3\},\\&\{3,4\}, \{4,5\}, \{5,7\},
        \{6,7\},
      \end{aligned}
      \\
      (1-\frac 1 p)(1-\frac 2 p), &I = \{3\},\\
      0, &\text{otherwise.}
    \end{cases}
  \end{equation*}
\end{lemma}

\begin{proof}
  We will see that
  \begin{equation}\label{eq:average_e8}
    \sum_{0 < \e_8 \le t} \theta_1'(\ee') =
    t\theta_2(\e_1,\dots,\e_7)+O(2^{\omega(\e_1\e_2\e_3\e_4\e_5\e_7)+\omega(\e_3)}),
  \end{equation}
  where
  \begin{equation*}
    \theta_2(\e_1, \dots, \e_7) = \sum_{\substack{k \mid \e_3 \\ \cp{k}{\e_2\e_4}}}
    \frac{\mu(k)\varphi^*(\e_3\e_4\e_5\e_6\e_7)}{k\varphi^*((\e_3,k\e_1))}
    \varphi^*(\e_1\e_2\e_3\e_4\e_5\e_7)
  \end{equation*}
  if (\ref{eq:cpe}) holds and $\theta_2(\e_1, \dots, \e_7)=0$
  otherwise.

  We observe that $\theta_1'(\ee') \in \Theta_{1,8}(3,\e_8)$
  \cite[Definition~3.8]{MR2520770} since we have $\theta_1'(\ee') \ll
  \prod_{i=1}^8 (\varphi^*(\e_i))^2 \in \Theta_{0,8}(0)$
  \cite[Definition~3.2]{MR2520770} by \cite[Example~3.3]{MR2520770},
  and because $\theta_1'(\ee')$ as a function in $\e_8$ lies in
  $\Theta_0(0)$ \cite[Definition~3.7]{MR2520770}
  by~(\ref{eq:average_e8}), and because its average is $\theta_2(\e_1,
  \dots, \e_7) \ll \prod_{i=1}^7 (\varphi^*(\e_i))^2 \in
  \Theta_{0,7}(0)$ as before, and because the error term is $\ll
  \prod_{i=1}^7 4^{\omega(\e_i)} \in \Theta_{0,7}(3)$ also as in
  \cite[Example~3.3]{MR2520770}.

  Furthermore, we see that $\theta_2(\e_1, \dots, \e_7)$ has the form of
  \cite[Definition~7.8]{MR2520770}, and a computation shows that its local
  factors $\theta_{2,p}$ are as in the statement of the result, so
  $\theta_2(\e_1, \dots, \e_7) \in \Theta_{4,7}'(2)$, and $\theta_2(\e_1,
  \dots, \e_7) \in \Theta_{2,7}(C)$ for some $C \ge 3$ by
  \cite[Corollary~7.9]{MR2520770}. In total, this shows $\theta_1'(\ee') \in
  \Theta_{2,8}(C)$ \cite[Definition~4.2]{MR2520770}.

  It remains to prove~(\ref{eq:average_e8}). If (\ref{eq:cpe}) does not hold,
  both sides are $0$. Otherwise,
  \begin{equation*}
    \sum_{0 < \e_8 \le t} \theta_1'(\ee')
    = \sum_{\substack{k \mid \e_3 \\ \cp{k}{\e_2\e_4}}}
    \frac{\mu(k)\varphi^*(\e_3\e_4\e_5\e_6\e_7)}{k\varphi^*((\e_3,k\e_1))}
    \sum_{\substack{0 < \e_8 \le t\\\text{(\ref{eq:cpe8})}}}
    \left\{\frac{-\e_2\e_4\e_7\e_8}{2^{v(k\e_1)}}\right\}.
  \end{equation*}
  We must show that the inner sum over $\e_8$ is $t\varphi^*(\e_1\cdots\e_5\e_7) +
  O(2^{\omega(\e_1\cdots\e_5\e_7)})$. Let $n=\min\{v(k\e_1),3\}$. If $n=0$, this
  holds by M\"obius inversion. If $n>0$, (\ref{eq:cpe}) implies that
  $\e_2,\e_4,\e_7$ are odd. Then the inner sum equals (with
  $\overline{-\e_2\e_4\e_7}$ the multiplicative inverse of $-\e_2\e_4\e_7$ mod $2^n$)
  \begin{equation*}
    \sum_{\substack{0 < \e_8 \le
        t\\\cp{\e_8}{\e_1\cdots\e_5\e_7}\\\congr{\e_8}{\overline{-\e_2\e_4\e_7}}{2^n}}}
    2^{n-1} = \sum_{l \mid \e_1\cdots\e_5\e_7} \mu(l) \sum_{\substack{0<\e_8'\le
        t/l\\\congr{l\e_8'}{\overline{-\e_2\e_4\e_7}}{2^n}}} 2^{n-1}.
  \end{equation*}
  If $l$ is even, the congruence is never fulfilled, so the inner sum over
  $\e_8'$ is $0$. If $l$ is odd, the inner sum over $\e_8'$ is
  $\frac{2^{n-1}t}{2^nl}+O(1)= \frac{t}{2l}+O(1)$. In total, the inner sum
  over $\e_8$ is
  \begin{equation*}
    \begin{split}
      \sum_{\substack{l \mid \e_1\cdots\e_5\e_7\\2 \nmid l}}
      \frac{\mu(l)}{2l}t + O(2^{\omega(\e_1\cdots\e_5\e_7)}) &=
      \frac{1}{2}t \prod_{\substack{p\mid \e_1\cdots\e_5\e_7\\p \ne 2}}
      \left(1-\frac 1 p\right) + O(2^{\omega(\e_1\cdots\e_5\e_7)})\\
      &=
      \varphi^*(\e_1\cdots\e_5\e_7)t+O(2^{\omega(\e_1\cdots\e_5\e_7)}),
    \end{split}
  \end{equation*}
  since $n>0$ implies that $\e_1\e_3$ is even.
  Summing the error term over $k$ only gives another factor $2^{\omega(\e_3)}$.
\end{proof}

Because of (\ref{eq:V1_bound}) and Lemma~\ref{lem:sum_e8}, we are in the
position to apply \cite[Proposition~4.3]{MR2520770}, giving
\begin{equation}\label{eq:final_summations}
  \sum_{\ee' \in \ZZ_1 \times \dots \times \ZZ_8} \theta_1'(\ee')V_1(\ee';B)
  = c_0V_0(B) + O(B(\log B)^5(\log \log B)^2)
\end{equation}
with
\begin{equation*}
  V_0(B) = \int_{\ee'} V_1(\ee';B) \dd \ee' = \int_{(\ee',\e_9) \in \R(B)}
  \e_1^{-1} \dd \e_9\dd \ee'
\end{equation*}
and
\begin{equation*}
  c_0 = \A(\theta_1'(\ee'),\e_8, \dots, \e_1) = \A(\theta_2(\e_1, \dots,
  \e_7), \e_7, \dots, \e_1) = \prod_p \omega_p,
\end{equation*}
whose local factors can be computed from the presentation of $\theta_2$ in
Lemma~\ref{lem:sum_e8} by \cite[Corollary~7.10]{MR2520770} as
\begin{equation}\label{eq:p-adic_density}
  \omega_p = \left(1-\frac 1 p\right)^7 \left(1+\frac 7 p + \frac 1{p^2}\right).
\end{equation}

Recall the definition~(\ref{eq:intervals}) of $J_i'$. We define
\begin{equation*}
  \begin{split}
    &\R_1'(B) = \left\{(\e_1, \dots, \e_6) \in J_1' \times \dots \times J_6' \midd
      \begin{aligned}
        &\base 32432000 \le B,\\
        &\base 53642{-2}00 \ge B
      \end{aligned}
    \right\},\\
    &\R_2'(\e_1, \dots, \e_6;B) = \{(\e_7,\e_8,\e_9) \in J_7' \times J_8' \times J_9' \mid h(\ee',\e_9;B) \le 1\},\\
    &\R'(B) = \left\{(\ee',\e_9) \in \RR^9 \midd
      \begin{aligned}
        &(\e_1, \dots, \e_6) \in \R_1'(B),\\
        &(\e_7,\e_8,\e_9) \in \R_2'(\e_1, \dots, \e_6;B)
      \end{aligned}
      \right\},\\
    &V_0'(B) = \int_{(\ee',\e_9) \in \R'(B)} \e_1^{-1} \dd\e_9\dd\ee',
  \end{split}
\end{equation*}
where the definition of $\R'_1(B)$ is inspired by the description of
the polytope whose volume is $\alpha(S)$ in~(\ref{eq:alpha_polytope}).

We claim that
\begin{equation}\label{eq:new_region}
  V_0(B) = V_0'(B) + O(B(\log B)^5).
\end{equation}
Comparing their definitions, in particular $J_i$ and $J_i'$ for $i \in
\{6,8\}$, we see that we must remove the conditions $\e_6 \ge 1$ and $|\e_8|
\ge 1$ and add the two conditions from the definition of $\R_1'(B)$, all with
a sufficiently small error term. We do this in four steps as in
\cite[Lemma~8.7]{MR2520770}; the order is important:
\begin{enumerate}
\item Add $\base 3 2 4 3 2 0 0 0 \le B$: This does not change anything
  because this condition follows from $\e_7 \ge 1$ and $\base 3 2 4 3
  2 0 1 0 \le B$ by~(\ref{eq:height}).
\item Add $\base 53642{-2}00 \ge B$: Using
  \cite[Lemma~5.1(3)]{MR2520770} for the integration over
  $\e_7,\e_9$, we see that the error term is
  \begin{equation*}
    \ll \int \frac{B^{5/6}}{|\base{1/6}{1/2}{0}{1/3}{2/3}{4/3}{0}{7/6}|} \dd(\e_1, \dots, \e_6,\e_8).
  \end{equation*}
  Using the opposite of our new condition for the integration over
  $\e_6$ together with $1 \le \e_1, \dots, \e_5 \le B$ and $|\e_8|\ge
  1$, we see that this is $\ll B(\log B)^5$.
\item Remove $|\e_8|\ge 1$: Using \cite[Lemma~5.1(1)]{MR2520770} for
  the integration over $\e_9$, we see that the error term is
  \begin{equation*}
    \ll \int \frac{B^{1/2}}{\e_1^{1/2}\e_2^{1/2}|\e_8|^{1/2}} \dd\ee'.
  \end{equation*}
  Using $|\e_8|\le 1$, and $\base 32432010 \le B$ for
  $\e_7$, and $\base 53642{-2}00 \ge B$ for $\e_6$, and finally $1 \le
  \e_1, \dots, \e_5 \le B$, we see that this is $\ll B(\log B)^5$.
\item Remove $\e_7 \ge 1$: Using \cite[Lemma~5.1(2)]{MR2520770} for
  the integration over $\e_8,\e_9$, we see that the error is
  \begin{equation*}
    \ll \int \frac{B^{3/4}}{\base{1/4}{1/2}0{1/4}{1/2}1{3/4}0} \dd(\e_1, \dots, \e_7).
  \end{equation*}
  Using $0 \le \e_7 \le 1$ and $\base 32432000 \le B$ for $\e_5$ with
  $1 \le \e_1, \dots, \e_4, \e_6 \le B$, we see that this is $\ll B(\log
  B)^5$.
\end{enumerate}

Next, we claim as in \cite[Lemma~8.6]{MR2520770} that
\begin{equation}\label{eq:alpha_real_density}
  V_0'(B) = \alpha(S)\omega_\infty B(\log B)^6.
\end{equation}
Indeed, substituting
\begin{equation*}
  x_2=B^{-1}\base 32432010,\ x_1=B^{-1}\base 11222221,\ x_3=B^{-1}\base
  01111111 \e_9
\end{equation*}
into $\omega_\infty$ as in Theorem~\ref{thm:main}, where $\e_1, \dots,
\e_6$ should be regarded as parameters and $\e_7,\e_8,\e_9$ as
the new integration variables, we see that
\begin{equation*}
  \frac{B\omega_\infty}{\e_1\cdots\e_6} = \int_{(\e_7,\e_8,\e_9) \in \R_2'(\e_1, \dots, \e_6;B)} 
  \e_1^{-1} \dd(\e_7,\e_8,\e_9).
\end{equation*}
Finally, we see that
\begin{equation*}
  \alpha(S)(\log B)^6 = \int_{R_1'(B)} \frac{1}{\e_1\cdots\e_6} \dd(\e_1,\dots, \e_6)
\end{equation*}
by substituting $\e_i = B^{t_i}$ into $\alpha(S) = \vol(P') = \int_{\tt \in
  P'} \dd \tt$ (see (\ref{eq:alpha_polytope}) below).

Combining Lemma~\ref{lem:first_sum2} with~(\ref{eq:final_summations}),
(\ref{eq:p-adic_density}), (\ref{eq:new_region}) and~(\ref{eq:alpha_real_density})
completes the proof of Theorem~\ref{thm:main}.

\subsection{Compatibility with Manin's conjecture}\label{sec:compatibility}

As the rank of $\Pic(\tS)$ is equal to $7$ (see Section~\ref{sec:geometry}), the exponent of $\log B$ in
Theorem~\ref{thm:main} is as predicted by Manin's conjecture. By
\cite{MR1340296}, \cite{MR1679843}, we have conjecturally $c_{S,H} = \alpha(S)
\cdot \omega_H(S)$.

We have
\begin{equation*}
  \alpha(S) = \frac{\alpha(S_0)}{\#W(\Afive) \cdot \#W(\Aone)}
  = \frac{1}{180\cdot 6! \cdot 2!} = \frac{1}{172800}
\end{equation*}
by \cite[Table~1]{MR2318651} and \cite[Theorem~1.3]{MR2377367}, where
$S_0$ is a split smooth cubic surface. Since
\begin{equation*}
  \begin{split}
    [-K_\tS] &= [3E_1+2E_2+4E_3+3E_4+2E_5+E_7],\\
    [E_8]&=[2E_1+E_2+2E_3+E_4-2E_6-E_7],
  \end{split}
\end{equation*}
we also have $\alpha(S) = \vol(P)=\vol(P')$, where
\begin{equation}\label{eq:alpha_polytope}
  \begin{split}
    P &= \left\{(t_1, \dots, t_7) \in \RRnn^7 \midd
      \begin{aligned}
        &3t_1+2t_2+4t_3+3t_4+2t_5+t_7 = 1,\\ 
        &2t_1+t_2+2t_3+t_4-2t_6-t_7 \ge 0
      \end{aligned}
    \right\}\\
    \cong P'&=\left\{(t_1, \dots, t_6) \in \RRnn^6 \midd
      \begin{aligned}
        &3t_1+2t_2+4t_3+3t_4+2t_5 \le 1,\\
        &5t_1+3t_2+6t_3+4t_4+2t_5-2t_6 \ge 1
      \end{aligned}
    \right\}.
  \end{split}
\end{equation}

Furthermore,
\begin{equation*}
\omega_H(S) = \omega_\infty \prod_p \left(1-\frac 1 p\right)^7 \omega_p,
\end{equation*}
where
\begin{equation*}
  \omega_p = \frac{\#\tS(\FF_p)}{p^2} = 1+\frac 7 p + \frac 1{p^2}
\end{equation*}
because the minimal desingularization $\tS$ of $S$ is a blow-up of $\Ptwo$
(which has $p^2+p+1$ points over $\FF_p$) in six points (each replacing one
point by an exceptional divisor containing $\#\PP^1(\FF_p) = p+1$ points over
$\FF_p$).

We check using the techniques of \cite{MR1340296}, \cite{MR1679843}
that $\omega_\infty$ is as in Theorem~\ref{thm:main} since the Leray
form of $\tS$ is
\begin{equation*}
  \omega_L(\tS) = (x_1x_2)^{-1} \dd x_1 \dd x_2 \dd x_3
\end{equation*}
(where $x_1x_2$ is the derivative of~(\ref{eq:surface}) with
respect to $x_0$) and by writing $x_0$ in terms of $x_1,x_2,x_3$ using
the defining equation~(\ref{eq:surface}).

\bibliographystyle{alpha}

\bibliography{manin_dp3_a1a5_2}

\end{document}